\documentclass[11pt]{article}
\usepackage{url}

\usepackage[latin1]{inputenc}
\usepackage[T1]{fontenc}
\usepackage{amsmath,amssymb,amsthm}
\usepackage{enumerate}
\usepackage{cite}

\DeclareMathOperator{\coker}{coker}
 \DeclareMathOperator{\HH}{H}
 
\DeclareMathOperator{\Hom}{Hom}
 \DeclareMathOperator{\Ext}{Ext}

\DeclareMathOperator{\Pic}{Pic}

\newtheorem{theorem}{Theorem}[section]
\newtheorem{lemma}[theorem]{Lemma}

\newtheorem{proposition}[theorem]{Proposition}

\newtheorem{remark}[theorem]{Remark}
\newtheorem{conjecture}[theorem]{Conjecture}

\usepackage{latexsym}            

\newcommand{\pp}{{\mathbb P}}

\newcommand\sE{{\mathcal E}}
\newcommand\sF{{\mathcal F}}

\newcommand\sH{{\mathcal H}}
\newcommand\sI{{\mathcal I}}

\newcommand\sL{{\mathcal L}}

\newcommand\sN{{\mathcal N}}
\newcommand\sO{{\mathcal O}}

\def\f{\Gamma}

\def\PP{{\mathbb P}}

\def\ZZ{{\mathbb Z}}

\def\PP{{\mathbb P}}

\newcommand{\proj}[1]
{ \mathchoice
           { {\mathbb P}^{#1} }
           { {\mathbb P}^{#1} }
           { {\mathbb P}^{#1} }
           { {\mathbb P}^{#1} }
         }

\begin{document}
\title{The Hilbert scheme of space curves sitting on \\  a 
  smooth surface containing a line}

\author{Jan O. Kleppe}

\date{}
\maketitle
\vspace*{-0.4in}
\begin{abstract}
  \noindent We continue the study of maximal families $W$ of the Hilbert
  scheme, $ \HH(d,g)_{sc}$, of smooth connected space curves whose general
  curve $C$ lies on a smooth degree-$s$ surface $S$ containing a
  line. 
  For $s \ge 4$, we extend the two
  ranges 
  where $W$ is a unique irreducible (resp.\; generically smooth) component
  of $ \HH(d,g)_{sc}$. In another range, close to the boarder of the nef cone,
  we describe for $s=4$ and 5 components $W$ that are non-reduced, leaving
  open the non-reducedness of only $3$ (resp. $2$) families for $s \ge 6$
  (resp. $s=5$), thus making progress to recent results of Kleppe and Ottem in
  \cite{KleOtt}. 
  For $s=3$ we slightly extend previous results on a conjecture of non-reduced
  components, and in addition we show its existence in a subrange of the
  conjectured range.
  \\[-2.5mm]

  \noindent {\bf AMS Subject Classification.} 14C05 (Primary), 14C20, 14K30,
  14J28, 14H50 (Secondary).


  \noindent {\bf Keywords}. Space curves, Quartic surfaces, Cubic surfaces,
  Hilbert scheme, Hilbert-flag scheme, relative Picard scheme
 \end{abstract}
 \vspace*{-0.25in}
\thispagestyle{empty}

\section{Introduction} 

Let $\HH(d,g)_{sc}$ be the Hilbert scheme of smooth connected curves in
$\proj{3}$. In this paper we study irreducible, possibly non-reduced,
components of $ \HH(d,g)_{sc}$ whose general curve sits on a smooth
surface containing a line. 
Since the first example of a non-reduced component was found by Mumford
\cite{Mu} there has been many geometers who were challenged by this phenomena
and quite a lot of papers has appeared that consider the problem of
non-reducedness and related questions, see e.g. \cite{DanMu,DP, E, F, GP2},
\cite{K1, K2,K3,K4, Krao}, \cite{MDP1,MDP3}, \cite{MuNa, N, Nanew} and
the books \cite{L1} and \cite{H2}.

A classical way of analyzing whether the closure of a family is a component,
possibly non-reduced, is to take a general curve $C$ of a
family 
and describe the curves in an open neighborhood of $(C) \in \HH(d,g)_{sc}$.
More recently several authors has been able to sufficiently describe the
obstruction of deforming $C$ and conclude similarly \cite{F,MDP3,Krao,N, MuNa,
  Nanew}. In the recent paper \cite{KleOtt} we find non-reduced, as well as
generically smooth, irreducible components of $ \HH(d,g)_{sc}$, and we prove
non-reducedness along the classical line. This works quite well if the genus
is large 
and the minimal degree $s(C)$ of a surface containing a general curve $C$ is
small (as in \cite{GP2}).

The families we consider are $s$-maximal families. To define them let
$W$ is an irreducible closed subset of $ \HH(d,g)_{sc}$ and let $s(W):=s(C)$
where $C$ is a general curve of $W$. As in \cite{K1} we say $W$ is {\em
  $s(W)$-maximal} if it is maximal with respect to $s(W)$, i.e. $s(V) > s(W)$
for any closed irreducible subset $V$ properly containing $W$. If $d >s^2$, an
$s$-maximal family containing $(C)$ is nothing but the image under the
forgetful morphism $pr_1:{\rm D}(d,g;s)_{sc} \to {\rm H}(d,g)_{sc}$, $(C,S)
\mapsto (C)$ of an irreducible component of the Hilbert-flag scheme $ {\rm
  D}(d,g;s)_{sc}$ containing $(C,S)$, see Section 2 for details.

An $s$-maximal family $W$ needs not be a component of $
\HH(d,g)_{sc}$. Indeed $4d \le \dim W$ is obviously a necessary condition for
$W$ to be a component, while $H^1(\sI_C(s))= 0$, $\sI_C$ the sheaf ideal of $C
\subset \PP^3$, turns out to be a sufficient condition because $pr_1$ is
smooth at $(C,S)$. When $H^1(\sI_C(s))= 0$, $W$ is even generically smooth if
the corresponding component of ${\rm D}(d,g;s)_{sc}$ is, e.g. if $s \le
4$. Since the dimension $\dim W$ of an $s$-maximal family is easy to compute
for $s \le 4$, we get in particular that $g \ge g_1:= 3d-18$ is necessary (for
$d > 9$ and $S$ smooth) while $g > g_2:=\lfloor (d^2-4)/8 \rfloor$ is
sufficient for a $3$-maximal family $W$ to be an irreducible component. Indeed
$H^1(\sI_C(3))= 0$ holds for $g > g_2$ by \cite[Cor.\,17]{K1}. Thus $W \subset
\HH(d,g)_{sc}$ may be a non-reduced component only if $g_1 \le g \le g_2$. 
The author conjectured in \cite{K2} that $H^1(\sI_C(3))\ne 0, d \ge 14$ and
$g_1 \le g \le g_2$ imply that a $3$-maximal
family 
is a non-reduced irreducible component of $
\HH(d,g)_{sc}$ and proved it in some subrange. Ellia extended the range
substantially in \cite{E}, and pointed out that the conjecture is false unless
we also suppose $H^1(\sI_C(1))= 0$. 
Looking more closely to Ellia's proof in \cite{E}, we extend the modified
conjecture (assuming also $H^1(\sI_C(1))= 0$) further in Theorem~\ref{4.7}, 
and we prove the existence of such components in a range covering all cases 
where the conjecture is proven (Theorem~\ref{exis}).

All components so far (i.e. for $s = 3$) are irreducible components $W$ of $
\HH(d,g)_{sc}$ whose general curve sits on a smooth surface {\sf containing a
  line $\f$}. 
Motivated by the theory of the Noether-Lefschetz locus (NL), the components of
NL whose general surface $S$ is smooth and {\sf contains a line} are very
special (\cite{Gr,Vo, EP}). In Section 3 we explicitly describe $s$-maximal
families for $s \ge 4$ on such surfaces $S \supset \f$ with Picard number 2.
If $d > s^2$ and $C \equiv a\f+b(H-\f)$, $H$ a hyperplane section, is a
general curve of $W$ and $a>s-4$ (resp. $a>s$), then $W$ is a unique
irreducible (resp.\;generically smooth) component of $
\HH(d,g)_{sc}$ except possibly in two cases. There are natural maps $\alpha_C:
H^0(S,\sN_{S}) \to H^1(C,\sN_{C/S})$, where $\coker(\alpha_C)$ is the
obstruction space of $ {\rm D}(d,g;s)_{sc}$ at $(C,S)$, and $\gamma_C:H^0(\sN_C)
\to H^1(\sI_C(s))$ 
that infinitesimally determine $pr_1$ at $(C,S)$ (see \eqref{alphaN} below).
In \cite[Thms.\;1.1, 1.2, 7.3]{KleOtt} we computed $\coker(\alpha_C)$ and
hence $\dim V$, and we showed that $W$ is a unique (resp. generically smooth)
component under a little too restrictive assumptions, 
see Theorem~\ref{mainQs} for the new version and Remark~\ref{conjrem} for the
generalizations in this paper compared to \cite[Thms.\;1.1, 1.2, 7.3]{KleOtt}.
In particular, also for $s \ge 5$ there are only three very explicitly
described families (corresponding to $H^1(\sI_C(s))\ne 0$), that may
correspond to non-reduced components, otherwise $W$ is a generically
  smooth component in $\HH(d,g)_{sc} $. Thus we extend in
Theorem~\ref{mainQs} (i) and (ii) \cite{KleOtt} in a good number of cases. For
the three families we still lack sufficient information of the map $\gamma_C$
to state that all these components are non-reduced. We expect that they are
non-reduced components which we prove for $s=4$ and for one family for
$s=5$, 
and in general provided $\gamma_C$
is non-zero.

While we should have liked to understand the map $\gamma_C$ better, the other
important map $\alpha_C$ is under the assumption $H^1(\sI_C(s-4))=0$ given by
its partner for the relative Picard scheme, $\Pic$. In fact the rational
morphism $\pi:{\rm D}(d,g;s)_{sc}\ -\!\to \Pic$, $(C,S) \mapsto (\sO_S(C),S)$
is smooth at $(C,S)$ provided $H^1(\sI_C(s-4))=0$ (\cite[Rem.\;4.5]{SB}). The smoothness of the morphisms $pr_1$ and $\pi$, as
well as \cite[Prop.\,1.5]{EP})
are important in the proof of
Lemma~\ref{lemGrothcomp} and hence of the theorems. In particular the
comparison between $\HH(d,g)$ and $\Pic$ works very well under the assumption
$H^1(\sI_C(s)) = 0 = H^1(\sI_C(s-4))$ which we in Section 2 illustrate by
giving an example (Remark~\ref{exPic}) of a non-reduced component of $\Pic$.

We thank O. A. Laudal for interesting discussions on
deformations and 
Hilbert-flag schemes and J. C. Ottem for working together in \cite{KleOtt} which has inspired this paper.

\subsection {Notations and terminology} In this paper the ground field $k$ is
{\em algebraically closed of characteristic zero} (and equal to the complex
numbers 
when the concept ''very general'' is used).
A surface $S$ in $\proj{3}$ is a hypersurface, and a curve $C$ in $\proj{3}$
(resp. in $S$) is a {\it pure one-dimensional} subscheme of
$\proj{}:=\proj{3}$ (resp. $S$) with ideal sheaf $\sI_C$ (resp. $\sI_{C/S}$)
and normal sheaf $\sN_C = {\sH}om_{\sO_{\proj{}}}(\sI_C,\sO_C)$ (resp.
$\sN_{C/S} = \Hom_{\sO_{S}}(\sI_{C/S},\sO_C)$). 
If $\sF$ is a coherent
$\sO_{\proj{}}$-Module, we let $H^i(\sF) = H^i(\proj{},\sF)$, $h^i(\sF) = \dim
H^i(\sF)$, $\chi(\sF) = \Sigma (-1)^i h^i(\sF)$  and 
\ $s(C) = \min \{\, n \, \arrowvert \, h^0(\sI_C(n)) \neq 0 \} \, .$ We denote
by $\HH(d,g)$ (resp. $\HH(d,g)_{sc}$) the Hilbert scheme of (resp. smooth
connected) space curves of Hilbert polynomial $\chi(\sO_C(t))=d(C)t+1-g(C)$,
and we let $d :=d(C)$ and $g:=g(C)$. 
A curve $C$ is called {\it unobstructed} if $\HH(d,g)$ is smooth at the
corresponding point $(C)$. A curve in a small enough open irreducible subset
$U$ of $\HH(d,g)$ is called a {\it general} curve of $\HH(d,g)$. 
 A {\it generization} $C' \subset \proj{3}$ of $C \subset \proj{3}$
 in 
 $\HH(d,g)$ is the general curve of some irreducible subset of $\HH(d,g)$
 containing $(C)$. By an irreducible component of $\HH(d,g)$ we always mean a
 {\it non-embedded} irreducible component. 
 A member of a closed irreducible subset $V$ of ${\HH}(s)$ or $\HH(d,g)_{sc}$
 is called {\it very general} in $V$ if it is smooth and sits outside a
 countable union of proper closed subset of $V$.

\section{ Background}

In the following we recall some results from the the background section of
\cite{KleOtt} that extend ideas and results appearing in \cite{K1}, \cite{K4}
and \cite{K98} and that use the deformation theory developed by Laudal in
\cite{L1}; in particular the results rely on \cite[Thm.\;4.1.14]{L1}. Moreover
ideas in \cite{SB} and \cite{EP} are central.

\subsection {The Hilbert flag scheme and the relative Picard scheme}

Let ${\rm D}(d,g;s)$ 
be the Hilbert-flag scheme 
parameterizing pairs $(C,S)$ of curves 
$C$ contained in a degree $s$-surface $S \subset \proj{3}$ where $d$ and
$g$ are the degree and genus of $C$.
If $S$ is smooth then $\sN_{C/S} \simeq \omega_C \otimes \omega_S^{-1}$ and we
have a connecting homomorphism $\delta: H^0(\sN_S\arrowvert_C) \to
H^1(\sN_{C/S}) \simeq H^0(\sO_C(s-4))^{\vee}$ induced by the sequence $ 0
\to \sN_{C/S} \to \sN_{C} \to \sN_S\arrowvert_C \to 0$ of normal bundles. Let
$\alpha_C:=\delta \circ m$ be the composed map of the natural restriction
$m:H^0(\sN_S) \to H^0(\sN_S\arrowvert_C)$ with $\delta$, let $A^2:= \coker
\alpha$ and let $A^1$ be the tangent space of ${\rm D}(d,g;s)$ at $(C,S)$.
Then the tangent map $ A^1 \rightarrow H^0(\sN_C)$ of the $1^{st}$ projection,
\begin{equation} \label{proj1} ~ pr_1: {\rm D}(d,g;s) \longrightarrow {\rm
    H}(d,g)\ , \quad {\rm induced \ by} \quad pr_1((C_1,S_1))= (C_1)\, ,
\end{equation}
at $(C,S)$ fits into an exact sequence 
\begin{equation} \label{alphaN} 0 \to H^0(\sI_{C/S}(s)) \to A^1 \rightarrow
  H^0(\sN_C) \rightarrow H^1(\sI_C(s)) \rightarrow \coker \alpha_C \to
  H^1(\sN_{C}) \rightarrow H^1(\sO_{C}(s))\rightarrow 0 \,
\end{equation}
from which we deduce $\dim A^1-\dim A^2=(4-s)d+g + {s+3 \choose 3}-2$. Note
that $pr_1$ is a projective morphism (\cite[Thm.\;24.7]{H2}). By
\cite[Lem.\,A10]{K1} $pr_1$ is {\sf smooth} at $(C,S)$ if $ H^1(\sI_C(s))= 0$.
Moreover $A^2= \coker \alpha_C$ contains the obstructions of deforming the
pair $(C,S)$ (\cite[(2.6)]{K1})), and we have $ A^2=0$ for $s \le 4$ if $C$ is
smooth and connected and not a complete intersection (c.i.) in $S$ (by the
infinitesimal Noether-Lefschetz theorem if $s=4$ and because $\delta=0$ if
$s\le 3$). Let $d > s^2$. Restricting $pr_1$ to the open set ${\rm
  D}(d,g;s)_{sc}$ 
where the curves are smooth and connected, we get that an $s$-maximal family
$W$ of $ \HH(d,g)_{sc}$ containing $(C)$ is nothing but the image under $pr_1$
of an irreducible component of ${\rm D}(d,g;s)_{sc}$ containing $(C,S)$
(\cite[Def.\;1.24 and Cor.\;1.26]{K3}).

We also need to consider the Hilbert scheme, ${\HH}(s) \simeq \proj{{s+3
    \choose 3}-1}$, of degree-$s$ surfaces in $\proj{3}$, the $2^{nd}$
projection
$pr_2: {\rm
  D}(d,g;s) \longrightarrow {\HH}(s) , \quad {\rm induced \ by}
\quad pr_2((C_1,S_1))= (S_1)$,
and the relative Picard scheme, $\Pic$, over the open set in ${\HH}(s)$ of
smooth surfaces of degree $s$. There is a projection $p_2:
\Pic \to {\HH}(s)$, forgetting the invertible sheaf, and a rational map:
\begin{equation} \label{pi} \pi: {\rm D}(d,g;s)\ -\!-\!\rightarrow {\Pic} ,
  \quad {\rm induced \ by} \quad \pi((C_1,S_1))= ( \sO_{S_1}(C_1),S_1)
\end{equation} 
defined over the {\it open subset} $U \subset {\rm D}(d,g;s)$ given by pairs
$(C_1,S_1)$ where $C_1$ is Cartier on a smooth $S_1$. Obviously, if we
restrict to $U$ we have $p_2 \circ \pi=pr_2$. If $ H^1(S,\sO_{S}(C)) = 0$ then
$ \pi$ is {\sf smooth} at $(C,S)$ by \cite[Rem.\;4.5]{SB}. Indeed, $
H^1(S,\sL)= 0$, $\sL:= \sO_{S}(C)$, implies a surjective map $A^1 \to
T_{\Pic,\sL}$ between the tangent spaces of ${\rm D}(d,g;s)$ at $(C,S)$ and
$\Pic$ at $(\sL)$ and an injection $\coker \alpha_{C} \to \coker \alpha_{\sL}$
of their obstruction spaces 
where $\alpha_{\sL}$ is the composition of $\alpha_{C}$ with the connecting
homomorphism $H^1(\sN_{C/S}) \to H^2(\sO_S)$ induced from the exact sequence
$0 \to \sO_S \to \sO_S(C) \to\sN_{C/S} \to 0$ (cf.\;\cite[Thm.\,1]{EP},
\cite[Sect.\,4]{K4} and \cite{K98}). 
Noticing that $ H^1(S,\sO_{S}(C)) \simeq H^{1}(\proj{3},\sI_{C}(s-4))^{\vee}$,
we have

\begin{lemma} \label{lemGrothcomp} Let $S \subset \proj{3}$ be a smooth
  degree-$s$ surface, $H$ a hyperplane section, let $E$ and $C$ be curves
  on $S$ satisfying $C \equiv$ (i.e. linearly equivalent to) $eE +fH$ for some
  $e \ne 0, f \in \mathbb Z$.  \\[1mm]
 {\rm (i)} Suppose
$H^1(\sI_E(s-4)) = H^1(\sI_C(s-4))=0 \,.$ Then $ {\rm D}(d,g;s)$ is smooth
at $(C,S)$ if and only if $ {\rm D}(d(E),g(E);s)$ is smooth at $(E,S)$, and
we have \\[-3mm]
\[ \dim \coker \alpha_C + h^0(\sI_{C/S}(s-4)) = \dim \coker \alpha_E +
h^0(\sI_{E/S}(s-4))\] \\[-3mm]
{\rm (ii)} If $H^1(\sI_E(s))=0 =H^1(\sI_E(s-4)) = H^1(\sI_C(s-4))$ and ${\rm
  H}(d(E),g(E))_{sc}$ is a smooth irreducible scheme and $(E) \in {\rm
  H}(d(E),g(E))_{sc}$, then $ {\rm D}(d,g;s)$ is smooth at $(C,S)$ and every
$(C',S') \in {\rm D}(d,g;s)$ satisfying $C' \equiv eE' +fH'$ for some $(E',S')
\in {\rm D}(d(E),g(E);s)_{sc}$, $H'$ a hyperplane section of a {\sf smooth}
surface $S' \subset \PP^3$, belongs to the unique irreducible component of $
{\rm D}(d,g;s)$ containing $(C,S)$.
\end{lemma}

\begin{proof} The main points in the proof is that $\pi$ is smooth at $(C,S)$
  and that $ \alpha_{\sO_{S}(C)} = e \cdot \alpha_{\sO_{S}(E)}$
  (\cite[Prop.\,1.5 and Constr.\,1.8]{EP}) leads to an isomorphism of their
  cokernels, see \cite[Lem.\;2.2]{KleOtt}. 
\end{proof}

 \begin{remark}[{\sf Example of a non-reduced irreducible
     component of the 
     relative Picard scheme $\Pic$}] \label{exPic}
  \ \ \\[-4mm]
 
 If both $H^1(\sI_C(s))$ and $H^1(\sI_C(s-4))$ vanish, then the morphisms
 $pr_1$ of \eqref{proj1} and $\pi$ of \eqref{pi} are smooth at $(C,S)$. Using
 this we get that many properties of the Hilbert scheme ${\rm H}(d,g)_{sc}$ at
 $(C)$ are transferred to the relative Picard scheme $\Pic$ at
 $(\sO_{S}(C),S)$, and vice versa. For instance if we take the general curve
 $C$ of the non-reduced component $V \subset {\rm H}(14,24)_{sc}$ of Mumford,
 and a smooth general surface $S$ of degree 6 containing $C$, then
 $(\sO_{S}(C),S)$ is the general point of a non-reduced irreducible component
 $P$ of \ $\Pic$. This follows from the fact that smooth morphisms take
 generic (resp. smooth) points onto generic (resp. smooth) points. Of course
 we have to verify the assumptions;
$$H^1(\sI_C(s))= H^1(\sI_C(s-4))=0\,  \ \ \ {\rm  for}  \ \ \ s=6 $$ 
with $S$ smooth and $(C,S)$ a general point of \ $ {\rm D}(d,g;s)$. For the
general curve in Mumford's example it is known that the homogeneous ideal
$I(C)$ allows 4 minimal generators of degree $\{3, 6, 6, 6\}$ that correspond
to smooth surfaces by \cite{Cu} and that $H^1(\sI_C(v))=0$ for $v \notin \{3,
4, 5\}$. Hence we may take the general degree-$6$ surface containing $C$ to be
smooth. Thus we conclude that $P$ is a non-reduced irreducible component of \
$\Pic$, see \cite{CiLo,Dan,DanMu} for a comparison
with the Noether-Lefschetz locus.
\end{remark}

We will also need the following lemma 
(cf. \cite[Cor.\! II 3.8]{Lop} and see \cite[Thm.\! II 3.1]{Lop} for a proof).

\begin{lemma} \label{lope} (A. Lopez) Let $E \subset \proj{3}$ be a smooth
  irreducible curve, let $n \ge 4$ be an integer and suppose the degree of
  every minimal generator of the homogeneous ideal of $E$ is at most $n-1$.
  Let $S$ be a very general smooth surface of degree $n$ containing $E$ and
  let $H$ be a hyperplane section. Then $\Pic(S) \simeq \mathbb Z \oplus
  \mathbb Z$ and we may take $\{ \sO_{S}(H), \sO_{S}(E) \}$ as a $ \mathbb
  Z$-basis for $\Pic(S)$.
\end{lemma}

\subsection {On the maximum genus of space curves}

Finally let us recall the definition of the maximum genus, $G(d,s)$, of smooth
connected space curves of degree $d$ not contained in a surface of degree
$s-1$, cf.\;\cite{GP1}. By
definition, 
\begin{equation} \label{maxgen} G(d,s) = \max \{\, g(C)\, \arrowvert \,
  (C) \in \HH(d,g)_{sc} \ {\rm \ and} \ \ H^0(\sI_{C}(s-1))=0 \, \}.
\end{equation}
In the C-range, i.e. when $d > s(s-1)$, Gruson and Peskine showed in \cite{GP1}
that
\begin{equation} \label{maxgen2} G(d,s) = 1+ \frac d 2 \left( \frac d s +
    s -4 \right) - \frac{r(s-r)(s-1)}{2s} \ \ \ { \rm where } \ d+r \equiv
  0 \ { \rm mod } \ s \ \ {\rm for} \ \ 0 \le r < s,
\end{equation}
and that $g(C)=G(d,s)$ if and only if $C$ is directly linked to a plane curve
$E$ of degree $r$ by a c.i. of type $(s,f)$, $f:=(d+r)/s$. 
%
In the ``extended $C$-range'': $s(s-1) \ge d \ge s^2-2s+2$ where the upper
part of the $B$-range is included, letting $ \mu:=d-( s^2-2s+3)$, one knows
that (\cite[Thm.\,A]{GP3})
\begin{equation} \label{maxgen3} G(d,s) = s^3-5s^2+9s-6+
  \frac{\mu(\mu+2s-3)}{2} \ \ \ \ {\rm for} \ \ \ s-3 \ge \mu \ge 0 \ , \ \ {
    \rm and }
\end{equation}
\begin{equation} \label{maxgen4} G(d,s) =1+d(s-3) \ \ \ \ {\rm for} \ \ \ d =
  s^2-2s+2 \ , \ \ {\rm i.e. \ \ for} \ \ \ \mu=-1 \,.
\end{equation}
Moreover, if $g(C)=G(d,s)$, then $C$ is ACM with $s(C)=s$ in \eqref{maxgen3},
and $C$ is a zero-section of the null correlation bundle twisted by $s-1$ in
\eqref{maxgen4}.


\section{Components of ${\rm H}(d,g)_{sc}$ for $s \ge 4$, the surface contains
  a line}

Let $S$ be a smooth surface of degree $s \ge 4$ in $\PP^3$
defined by a form $x_0P+x_1Q$, where $P,Q$ are very general homogeneous
polynomials of degree $s-1$. Families of curves on such surfaces were studied
in Theorems 1.1, 1.2 and 7.3 of \cite{KleOtt}. Here we improve upon these
results.

Let $\f_1=\{x_0=x_1=0\}$ and $\f_2=\{x_0=Q=0\}$. The hyperplane section
satisfies $H \equiv \f_1+\f_2$, $H^2=s$ and we may suppose $\Pic(S) \simeq\ZZ
\f_1 \oplus \ZZ H$ by Lemma~\ref{lope} and that $\f_2$ is a smooth connected
curve. If $C \equiv a\f_1+b\f_2$ then $d=C \cdot H$, $K=(s-4)H$ and the
adjunction formula which implies that the intersection matrix is
$(\f_i\cdot\f_j)=\left(\begin{smallmatrix}2-s & s-1\\
    s-1 & 0 \end{smallmatrix}\right)$, leads also
to 
\begin{equation} \label{sge5}
 d=a+(s-1)b \ \ {\rm and} \ \ g=1+(s-1)ab +
\frac{1}{2}((s-4)a+(s-4)(s-1)b-(s-2)a^2)\, .
\end {equation}
\begin{lemma} \label{nefbp} 
  Let $S$ be a smooth surface of degree $s$ with $\f_1,\f_2$ as above. It holds:
\item {\rm (i)} Any effective divisor on $S$ is linearly equivalent to
  $a\f_1+b\f_2$ where 
$a,b\ge 0$. 
\item {\rm (ii)} Every nef divisor is linearly equivalent to $a\f_1+b\f_2$ where
  $(s-1)b\ge (s-2)a\ge 0$.
\item {\rm (iii)} The divisor $D=(s-1)\f_1+(s-2)\f_2$ is base-point free. 
\item {\rm (iv)} Any divisor $C$ satisfying $C \cdot \f_1 \ge 0$ and $C \cdot
  \f_2 \ge 0$ is nef and base-point free.
\item {\rm (v)} If a divisor $C$ satisfies $C \cdot \f_1 \ge
  0$ and $C \cdot \f_2 >0$ then $|C|$ contains a smooth irreducible curve.
\end{lemma}

See \cite{KleOtt} for a proof. Now the Kawamata-Viehweg vanishing theorem, for
$D \equiv a\f_1+b\f_2$, implies:

\begin{equation} \label{KVvan} H^i(S,\mathcal{O}_S(D))=0 \ \ {\rm for} \ i>0 \
  \  {\rm provided} \  \ a>s-4 \  {\rm and} \ (s-1)b \ge (s-2)a+s-4 
\end{equation}
because the assumptions on $a,b$ imply that $D-K$ is nef and big. From this,
we got Theorem 7.3 of \cite{KleOtt}. Here we generalize \eqref{KVvan} leading
to a significant improvement of 
that result, cf.\, \cite[Lem.\,2.5]{KleOtt}. 
Indeed we have 
\begin{lemma} \label{nefbp} Let $S$ be a smooth surface of degree $s$ with
  $\f_1,\f_2$ as above and let $C$ be a divisor linearly equivalent to
  $a\f_1+b\f_2$ where $(s-1)b - (s-2)a = t$ with $t \ge -2$. Moreover suppose
  $ a>s-4$. 
  Then
\item {\rm (i)} \ \ $H^i(S,\mathcal{O}_S(C))=0  \quad {\rm
    for} \quad  i>0   \quad {\rm provided} \quad   t > -2  \ {\rm
    or} \ ( t = -2 \ {\rm
    and}  \ a = s -3)   , \
  {\rm     and} $ 
\item {\rm (ii)} \ \  $H^1(S,\mathcal{O}_S(C)) \simeq k \ , \ \  H^2(S,\mathcal{O}_S(C)) = 0 \quad
{\rm provided} \quad t = -2 \ \ {\rm
  and} \ \ a \ne s -3 \ .$
\end{lemma}
\begin{proof} Firstly we suppose $a > s-3$. To apply \eqref{KVvan} onto
  $D=C-\f_1$, we notice that \[   (s-1)b = (s-2)a  - (s -2)+s -2 + t \ge (s-2)(a-1) 
 + s -4\] by assumption. Thus $H^i(S,\mathcal{O}_S(C - \f_1))=0$ for
  $i>0 $ by  \eqref{KVvan}. Moreover since 
$$0\to \mathcal{O}_S(C-\f_1)\to \mathcal{O}_S(C)\to
\mathcal{O}_{\f_1}(C\cdot\f_1)\to 0$$ is exact, we deduce the exact
sequence \begin{equation} \label{tneg} \to H^1(\mathcal{O}_S(C-\f_1)) \to
  H^1(\mathcal{O}_S(C))\to H^1(\mathcal{O}_{\PP^{1}}(C\cdot\f_1)) \to
  H^2(\mathcal{O}_S(C-\f_1)) \to H^2(\mathcal{O}_S(C))\to 0
\end{equation} 
where $C\cdot\f_1=-a(s-2)+b(s-1) = t$ and we easily get the lemma in the case
$a \ne s-3$.

Finally if $a=s-3$ we apply \eqref{KVvan} onto $D:=C$. Since we have $a > s-4$
it suffices to show $(s-1)b \ge (s-2)(s-3)+s-4$, or equivalently $b \ge
s-3-1/(s-1)$ which holds if $b \ge s-3$. Since $b \ge s-4$ follows from the
assumption $(s-1)b= (s-2)a + t \ge (s-2)(s-3)+(-2)$ and the curve $\f_1+(s-4)H
\equiv (s-3,s-4)$ is ACM, we get the lemma.
\end{proof}

The assumption $t \ge -2$ in Lemma~\ref{nefbp} may be weakened. Indeed we have

\begin{lemma} \label{nefbp2} 
  Let $S$ be a smooth surface of degree $s$ with $\f_1,\f_2$ as above, let
  $C \equiv a\f_1+b\f_2$ and $t:=(s-1)b - (s-2)a $ and suppose $-1 \ge t \ge
  -4$, $ a>s-2$, $s \ge 4$ and $ (t,s) \ne (-4,4)$ (resp. $(t,s)=(-4,4)$).
  Then $$\dim 
  H^1(S,\mathcal{O}_S(C))= 
  -t-1 \ ( \,{\rm resp.} \ 4 \ {\rm if} \ (t,s)=(-4,4)\,) \ . $$
  \end{lemma}

  \begin{proof} Suppose $ (t,s) \ne (-4,4)$. Then we have
    $H^i(S,\mathcal{O}_S(C - \f_1))=0$ for 
    $i>0$ by Lemma~\ref{nefbp} since $a-1>s-3$ and $(s-1)b - (s-2)(a-1)=t+s-2
    > -2$, i.e.
    the     assumptions of     the lemma are 
    fulfilled for $(a-1)\f_1+b\f_2$. By \eqref{tneg}, we get $
    h^1(\mathcal{O}_S(C)) = h^1(\mathcal{O}_{\PP^{1}}(C\cdot\f_1)) = -t-1$
    because  $C\cdot\f_1 = t$. If $(t,s)=(-4,4)$, then 
    Lemma~\ref{nefbp} yields  $H^1(S,\mathcal{O}_S(C - \f_1)) \simeq k$ and we
    conclude by \eqref{tneg}. 
\end{proof}

We are now ready to generalize \cite[Thms.\;1.1, 1.2, 7.3]{KleOtt}. Below we
often write  $(a,b)$ for $ a\f_1+b\f_2$. 

\begin {theorem} \label{mainQs} Let $S \subset \proj{3}$ be a smooth
  degree-$s$ surface containing a line $ \Gamma_1$, let $\Gamma_2 \equiv H -
  \Gamma_1$ 
  be a smooth curve and suppose $\Pic(S) \simeq\ZZ \f_1 \oplus \ZZ \f_2$ (e.g.
  $S$ is very general)  and
  $s \ge 4$. Let  $C \equiv a\f_1+b\f_2$ 
  be a smooth connected curve of degree $d >
  s^2$ 
  with  $a \ne b$.

\item {\rm (i)} Suppose $ a>s-4$. Then $C$ belongs
  to a unique $s$-maximal family $W \subset \HH(d,g)_{sc}$. Moreover if
  $\tilde S$ is a degree-$s$ surface containing a very general member of $W$,
  then $\Pic(\tilde S)$ is freely generated by the classes of a line and a
  smooth plane degree-$(s-1)$ curve, and every $C \equiv a\f_1+b\f_2$
  contained in 
  some surface $S$ as above belongs to $W$. Furthermore 
  \[\dim W =(4-s)d+g + {s+3 \choose 3}+ {s-1 \choose 3} -s+ 1 \ \ with \ d,g \
  as \ in \ \eqref{sge5}, \]  and if $(s,a,b) \notin
  \{(4,6,4), (4,9,6)\}$ then $W$ is an
  irreducible component of \ $\HH(d,g)_{sc}$. 
\item {\rm (ii)} Suppose $s < a < $ {\Large$\frac{(s-1)b-2}{s-2}$} or $(a,b)
  = (s+1,s)$. Then 
  $W$ is moreover a generically smooth  irreducible  component  of \
  $\HH(d,g)_{sc}$. 
\item {\rm (iii)} Suppose {\Large$\frac{(s-1)b-2}{s-2}$}$ \le a \le
  ${\Large$\frac{(s-1)b}{s-2}$},  $(a,b) 
  \ne (s+1,s)$ and $(s,a,b) \notin
  \{(4,6,4), (4,9,6)\}$. 
  Then $W$ is a non-reduced 
  irreducible component of \ $\HH(d,g)_{sc}$ provided $h^0(\sN_C) > \dim W$,
  or equivalently, provided the  map $H^0(\sN_C) \to H^1(\sI_C(s))$ appearing in
  \eqref{alphaN}   is non-zero for   the general curve $C$ of $W$.
  In particular $W$ is a non-reduced irreducible component of \
  $\HH(d,g)_{sc}$  if $s=4$, {\sf
    or} $s=5$ {\rm and} $(a,b)=(4k,3k)\, , \ k \geq 2$. 
\end{theorem}

\begin{remark} \label{conjrem} {\rm (A)} 
  Note that for any $C \equiv (a,b)$ of the theorem we have $a \le $
  {\Large$\frac{(s-1)b}{s-2}$} because $C$ is nef. 

  {\rm (B)} Since the assumption in {\rm (iii)} and $d > s^2$ imply $a > s$,
  there are exactly three families, \\[-2mm]

  \hspace{4cm}
  $(a,b)= ((s-1)n-\mu,(s-2)n-\mu+1)$, \ \ $n \ge 3$ \\[2mm]
  corresponding to {\rm (a)}: $\mu=s-3$, {\rm (b)}: $\mu = s-2$ and {\rm (c)}:
  $\mu=s-1$ respectively, that are covered by {\rm (iii)} of
  Theorem~\ref{mainQs} (where {\rm (c)} is on the border of the nef cone).
  Thanks to Lemma~\ref{nefbp2} we
  have \\[-2mm]
  
 \hspace{0.4cm} 
  $h^1(\sI_C(s))=1$ {\rm (}resp. $h^1(\sI_C(s))=2$, $h^1(\sI_C(s))=3$\,{\rm )}
  for the family  {\rm (a) (}resp.  {\rm (b), (c)\,)}\, \\[2mm]
  if $s \ge 5$ for family {\rm (c)}. We {\sf expect} that the
  corresponding irreducible components of \ $\HH(d,g)_{sc}$ are non-reduced.
  Indeed they are non-reduced for the family {\rm (c)} when $s=5$ by
  Theorem~\ref{mainQs}.

  {\rm (C)} If $s=4$ then Theorem~\ref{mainQs} implies that the three
  families described in {\rm (B)} form non-reduced components (with $n \ge 5$
  in {\rm (c)}). Thus we slightly generalize \cite[Thm.\,1.2 (II)]{KleOtt} by
  including two more cases of non-reduced components. Note that for family
  {\rm (c)}, $h^1(\sI_C(s))=4$ in this case.

  {\rm (D)} For $s=5$, Theorem~\ref{mainQs} {\rm (ii)}, resp. {\rm (iii)}
  slightly generalize (I), resp. (II) of \cite[Thm.\,1.2]{KleOtt} by including
  one more infinite family (resp. one more case of a non-reduced component).

  {\rm (E)} For $s \ge 6$ the generalization in Theorem~\ref{mainQs} of
  \cite[Thm.\,7.3]{KleOtt} is more substantial because Theorem~\ref{mainQs}
  {\rm (i)}, resp. {\rm (ii)} includes, as $s$ increases, an increasing number
  of infinite families for which 
  Theorem~\ref{mainQs} holds (compared with what holds by
  \cite[Thm.\,7.3]{KleOtt}).
\end{remark}


\begin{proof}
  The assumptions on $a,b$ in (i), resp. (ii), imply that $H^1(\sO_S(C))=0$,
  resp. $H^1(\sI_C(s))^{\vee} \simeq H^1(\sO_S(C-4H))=0$ by Lemma~\ref{nefbp}.
  We therefore get the stated properties of $W$ in (i), resp. (ii) by the same
  proof is that in \cite[Thm.\! 7.3]{KleOtt} (except for $W$ being a component
  in (i)); 
  the whole point is only that Lemma~\ref{nefbp} imply the vanishing of
  $H^1(\sO_S(C+vH))$ under weaker assumptions than those used in \cite[Thm.\!
  7.3]{KleOtt}, which lead to corresponding improvements of
  Theorem~\ref{mainQs} (i) and (ii). 
    
  In (i) it only remains to see that $W$ is an irreducible component. Observe
  that $H^1(\sI_C(s))=0$ does not only prove the generic smoothness of $W$,
  but it also implies that $W$ is an irreducible component of $ H(d,g)_{sc}$.
  There are, however, cases (i.e. the 3 families mentioned in (B) above) not
  covered by (ii) of the theorem, and for these we prove that $W$ is an
  irreducible component as in \cite[Thm.\,7.3]{KleOtt} by e.g. showing $g \ge
  G(d,s+1)$, except when $(s,a,b) \in \{(4,7,5), (4,12,8), (5,8,6),
  (6,10,8)\}$. In these four cases we were not able to show that $W$ was an
  irreducible component in \cite[Thms.\,1.1, 7.3]{KleOtt} because $g
  <G(d,s+1)$. To get (i) and (ii) as stated above it remains to consider them
  now.

  In these cases we suppose there is an irreducible component $V$ of
  $\HH(d,g)_{sc}$ satisfying $W \subset V$ and $\dim W < \dim V$. Since $W$ is
  $s$-maximal, we may suppose that the general curve $X$ of $V$ satisfies
  $s(X) > s$.

  {\bf The case} $(a,b)=(12,8)$ and $s=4$. For this class we compute the
  following numbers: $(d,g)=(36,145)$, $G(d,5)=147$ and $G(d,6)=145$,
  cf.\;\eqref{maxgen}. We first consider the option $s(X)=5$. To see that $X$,
  whence $V$ does not exist
  we use the theory on Halphen's gaps given in Ellia's paper \cite{E1}. If
  such a curve exists, we compute $r$ in $d(X)+r \equiv 0 \ { \rm mod } \ 5$,
  and we get $r=4=s(X)-1$. Using \cite[Prop.\;IV.3]{E1} it follows that there
  are no such $X$ with maximal numerical character, and by
  \cite[Lem.\;VI.2]{E1} that the genus of $X$ is equal to the genus of the
  numerical character of $X$. This implies that $X$ is arithmetically
  Cohen-Macaulay (ACM).

  By Riemann-Roch $\chi(\sI_X(5))= 20$, whence $h^1(\sO_X(5))=19$ and since
  $\chi(\sI_X(8))= 21$, we get $h^1(\sO_X(8)) \le 1$. It follows that $ \dim\,
  _0\!\Hom_R(I(X), H^1_{*}(\sO_X))=19$. Using
  \begin{equation} \label{maxrank} h^0(\sN_X) = 4d + \ \dim\!\,
    _{(-4)}\!{\Hom_R} (I(X), H^1_{*}(\sI_X))+ \dim\, _0\!\Hom_R (I(X),
    H^1_{*}(\sO_X)) \
\end{equation}
which holds for curves of maximal rank (cf.\;Remark~\ref{maxrankrem}), in
particular for ACM curves (where we have $\Hom_R (I(X), H^1_{*}(\sI_X))=0$),
we conclude that $h^0(\sN_X)=163 < g+33=178$ which contradicts the assumption
$\dim W < \dim V$.

Also the case $s(X) \ge 6$ must be considered. Since $g = G(d,6)=145$ and
$d(X) \equiv 0 \ { \rm mod } \ 6$, $X$ is a c.i. of type $(6,6)$ by
\eqref{maxgen2}. Since $\chi(\sI_X(6))= 12$ we get $h^1(\sO_X(6)) = 10$ and $
\dim_{(X)}\HH(d,g) = 4d+2h^1(\sO_X(6))=164$, a contradiction. This shows that
$W$ is a component.

{\bf The case} $(a,b)=(7,5)$  and $s=4$. We compute the following numbers:
$(d,g)=(22,57)$, $\chi(\sI_C(5))= 2$, $\chi(\sI_C(6))= 8$ and $G(d,5)=58$.
Note that $g=G(d,5)-1$; whence we suppose a generization $X$ of $C$ satisfying
$s(X)> 4$ exists. Since $H^1(\sO_C(6))=0$ by the speciality theorem,
(cf.\;\cite{GP1}), 
we get $H^1(\sO_X(6))=0$ by semicontinuity and in particular $h^0(\sI_X(6))
\ge 8$ and $s(X) \le 6$. 

Firstly suppose $s(X)= 6$. In this case $X$ belongs to the so-called $B$-range
in the classification of curves of maximal genus. Since $\chi(\sI_X(5))= 2$,
we get $h^1(\sO_X(5)) \ge 2$. Then \cite[Prop.\;3.2]{Har} implies that $d \ge
A(6,5)=23$, a contradiction (see \cite{E1} and \cite{Har} for the definition
of $A(k,f)$ and a discussion on the maximum genus when $d \le s(s-1)$). Or one
may use that the maximum genus $G(d,s)$ in range $B$ is known in our case
(cf.\;\cite{Har}, Cor.\;3.8 and Rem.\;3.8.1); it is $G(22,6)=55$, as
conjectured in \cite[Conj. 3.5]{Har}, whence such an $X$ does not exist.

Secondly suppose 
$s(X)= 5$. Since $d(X)+r \equiv 0 \ { \rm mod } \ s(X)$ allows $r=3=s(X)-2$,
it follows from \cite[Prop.\;IV.4]{E1} that $X$ is bilinked via c.i.'s of type
$(5,8)$ and $(5,4)$ to a degree-2 curve $Z$ satisfying $g(Z)=-1$ provided the
numerical character is maximal. By \cite[Lem.\;V.1]{E1} the minimal resolution
of $I(Z)$ is known (the resolution of two skew lines have the same Betti
numbers) and the mapping cone construction used twice yields:
\begin{equation*} 0 \rightarrow R(-8) \rightarrow R(-9) \oplus R(-8) \oplus
  R(-7)^{4} \rightarrow R(-8) \oplus R(-6)^4 \oplus R(-5) \rightarrow I(X)
  \rightarrow 0\, .
\end{equation*} 
Thus $X$ is of maximal rank and applying \eqref{maxrank}, we get $ h^0(\sN_X)
\le 4d+2=90$ because $\dim H^1_{*}(\sI_X)=h^1(\sI_X(4))=1$ and
$h^1(\sO_X(5))=1$. Since $\dim W = g+33 =90$, we get a contradiction.

Thirdly if the numerical character is not maximal we can again use
\cite[Lem.\;VI.2]{E1} (or \cite{Dol}) to see that the genus of $X$ is equal to
the genus of the numerical character of $X$. This implies that $X$ is ACM and
using \eqref{maxrank}, we get $ h^0(\sN_X) \le 4d+1=89$, i.e. a contradiction.

{\bf The case} $(a,b)=(8,6)$ and $s=5$. We compute the following numbers:
$(d,g)=(32,113)$, and $G(d,6)=115$. Note that $g=G(d,6)-2$. Firstly suppose
$s(X) \ge 7$. Then $X$ belongs to the $B$-range in the classification of
curves of maximal genus. The conjecture in \cite[Conj.\,3.5]{Har} is true for
$s(X) \le 9$ by \cite[Rem.\;3.8.1]{Har}. To find the conjectured value of
$G(d,7)$ we compute $A(7,f)$ for several $f$ to find the largest $f$
satisfying $A(7,f) \le d$ (see \cite{E1} or \cite[p.\,530]{Har} for the
definition of $A(s,f)$ and $B(s,f)$). We find $A(7,7)=33$, $A(7,6)=28$ and
$B(7,6)=31$, whence $G(d,7)=111$ by \cite[Conj.\,3.5]{Har} and such a
generization $X$ of $C$ does not exists because $g=113$.

  Therefore we suppose 
  $s(X)= 6$. Since $d(X)+r \equiv 0 \ { \rm mod } \ s(X)$ allows $r=4=s(X)-2$,
  it follows from Ellia's paper \cite[Prop.\;IV.4]{E1} that $X$ is bilinked
  via c.i.'s of type $(6,10)$ and $(6,5)$ to a degree-2 curve $Z$ satisfying
  $g(Z)=-2$ provided the numerical character is maximal. One may use
  \cite[Prop.\,V.2]{E1} to see $ h^0(\sN_Z) = \dim_{(Z)} \HH(d,g)=9$ and
  \cite[Cor.\,3.6]{K3} to compute the dimension of the bilinked family. Let us
  instead use that the minimal resolution of $I(Z)$ is known
  (\cite[Lem.\;V.1]{E1}). Then the mapping cone construction used twice yields
  a resolution:
  \begin{equation*} 0 \rightarrow R(-10) \rightarrow R(-9)^2 \oplus R(-8)^2
    \oplus R(-11) \rightarrow R(-8) \oplus R(-7)^3 \oplus R(-6) \rightarrow
    I(X) \rightarrow 0\, 
    \end{equation*} 
    where we have skipped a redundant term ($R(-10)$ ``in the middle'') since
    we may use the same surface of degree 6 in both linkages. This is a curve
    of maximal rank. To apply \eqref{maxrank}, we compute the following
    numbers using well known liaison formulas (e.g. \cite[(2.18.1)]{K3}):
    $h^1(\sI_X(4))=h^0(\sI_Z(-1))=1$, $h^1(\sO_X(6))=4+h^1(\sO_Z(1))=4$,
    $h^1(\sO_X(7))=1+h^1(\sO_Z(2))=1$ and that $H^1(\sI_X(v)) = 0$ for $v
    \notin \{4,5,6\}$. Hence $ \dim\, _0\!\Hom_R (I(X), H^1_{*}(\sO_X)) \le 7$
    and then \eqref{maxrank} implies $ h^0(\sN_X) \le 4d+7+1=136$. Since $\dim
    W =-d+g + 56 =137$ by the proven part of Theorem~\ref{mainQs}, we get a
    contradiction.

    If the numerical character is not maximal we can again use
    \cite[Lem.\,VI.2]{E1} (or \cite{Dol}) to see that the genus of $X$ is
    equal to the genus of the numerical character of $X$. This implies that
    $X$ is smooth and ACM. By Riemann-Roch, $\chi(\sI_X(6))= 4$ and
    $\chi(\sI_X(7))= 8$. If $h^0(\sI_X(6))=1$, then $h^1(\sO_X(6))=3$,
    $h^1(\sO_X(8))=0$ and we get at least $ \dim\, _0\!\Hom_R (I(X),
    H^1_{*}(\sO_X)) \le 7$ (by looking at the options given by
    $h^0(\sI_{X/S}(7)):=q$, $h^1(\sO_X(7))=4-q$, $0 \le q \le 4$ where $S$ is
    defined by a degree-6 polynomial), and if $h^0(\sI_X(6))=2$, then a
    linkage via a c.i. of type $(6,6)$ yields an ACM curve of degree 4 and
    genus 1, and one shows $ \dim\, _0\!\Hom_R (I(X), H^1_{*}(\sO_X))= 4$. In
    any case we get $ h^0(\sN_X) \le 135$ by \eqref{maxrank}, i.e. a
    contradiction, and the proof of this case is complete.

    {\bf The case} $(a,b)=(10,8)$ and $s=6$. We compute the following numbers:
    $(d,g)=(50,251)$, and $G(d,7)=252$. cf.\;\eqref{maxgen}. Note that
    $g=G(d,7)-1$. Let $X$ be a generization of $C$ satisfying $s(X)> 6$. We
    first consider the option $s(X)=7$. If such a curve exists, we have
    $d(X)+r \equiv 0 \ { \rm mod } \ 7$, i.e. $r=6=s(X)-1$. It is tempting to
    say that $(d,g)$ is a known Halphen's gap (cf.\,\cite{Dol}), but before we
    can do that we have to compute $G(d,8)$. We have $t^2-2t+2=d$ if $t=8$,
    and in this part of the $B$-range (or extended $C$-range) it is known that
    $G(d,8)=1+d(t-3)= 251$ by \eqref{maxgen4}. So strictly speaking, since
    $g=G(d,s(X)+1)$, it is not an Halphen's gap. We can, however, still use
    Ellia's results in \cite{E1} to show that $X$ does not exists. Indeed by
    \cite[Prop.\,IV.3]{E1} it follows that there are no such $X$ with maximal
    numerical character. Since \cite[Lem.\,VI.2]{E1} implies that the
    numerical character of $X$ had to be maximal, such an $X$ does not exist.
    
    Finally we suppose $s(X)\ge 8$. Since $g=G(d,8)$, it follows from
    \cite{GP3}, cf.\,\eqref{maxgen4} and \cite{E2}, that $X$ is
    a zero section of the null correlation bundle $\sE$; more precisely there
    is an exact sequence
\begin{equation} \label{serre} \ \ 0 \longrightarrow \sO_{\PP^3}
  {\longrightarrow} \sF \rightarrow \sI_{X}(14) \longrightarrow 0 \ 
\end{equation} 
where $ \sF= \sE(7)$. To compute the dimension of the component $V$ of $
\HH(d,g)_{sc}$ to which $X$ belongs we will use the formula appearing in
\cite[Cor.\;2.3]{Krefl} (see also Ellia's joint work with Fiorentini
\cite{EF}), stating that if we have \eqref{serre} and $H^1(\sI_X(c_1)) = 0 =
H^1(\sI_X(c_1-4))$, $c_i:=c_i(\sF)$ the $i$-th Chern class, then $\sF$ is a
smooth point of its moduli space ${\rm M_{\PP^3}}(c_1,c_2,c_3)$ if and only if
$X$ is unobstructed and
\begin{equation} \label{serre1}
  {\dim_{(\sF)}{\rm M_{\PP^3}}}(c_1,c_2,c_3)+h^0(\sF) =
 \dim_{(X)}\HH(d,g) + h^0(\omega_X(-c_1+4))\ .
\end{equation} 
 Using this formula for $ \sF=
 \sE(7)$,
 remarking that $H^1(\sE(v))=0$ for $v \ne -1$ is well known (since a section
 of $\sE(1)$ corresponds to two skew lines), we see that the assumptions for
 \eqref{serre1} to hold are fulfilled. Since $\dim {\rm M_{\PP^3}}(0,1,0)=5$
 and we have $h^0(\omega_X(-10))=1$ and $h^0(\sF)=h^0(\sE(7))=231$ by
 Riemann-Roch, we get $\dim V = \dim_{(X)}\HH(50,251) =235$. Finally using
 that $\dim W =-2d+g + 84+10-5=240 > \dim V$, we get a contradiction and the
 proof of (i) is complete.

 (iii) The component $W$ is non-reduced if we can show $\dim W < h^0(\sN_C)$
 for $C$ general. Since $\dim W = \dim A^1$, $H^0(\sI_{C/S}(5)) = 0$ and $\dim
 \coker \alpha_C = 2$ it suffices by \eqref{alphaN} to prove $h^1(\sI_C(5))
 \ge 3$. This follows from Lemma~\ref{nefbp2}, and proof of
 Theorem~\ref{mainQs} is complete.
\end{proof}

\begin{remark} \label{maxrankrem} Suppose that $X$ has maximal rank, or more
  generally that \ $_{0}{\Ext^i_R} (I(X), H^1_{*}(\sI_X))= 0$ for $i=$ 0, 1
  and 2. Then the formula \eqref{maxrank} follows easily from the exact
  sequence \cite[(2.4)]{Krao} because \cite[(2.1)]{Krao} implies
  ${_0\!\Ext_{\mathfrak m}^2}(I(X),I(X))=0$ and ${_0\!\Ext_{\mathfrak
      m}^3}(I(X),I(X)) \simeq \ _0\!\Hom_R (I(X), H^1_{*}(\sO_X))$ and the
  duality in \cite[(2.2)]{Krao} implies $ {_0}\!\Ext_{R}^2(I(X),I(X))^{\vee}
\simeq \ _{-4}\!\,{\Hom_R} (I(X), H^1_{*}(\sI_X))$.

If $X$ is ACM, or more generally if $ {_0}\!\Ext_{R}^2(I(X),I(X))=0$ and $X$
is of maximal rank, 
then $X$ is unobstructed, and we may replace $ h^0(\sN_X)$ by $\dim_{(X)}
\HH(d,g)$ in \eqref{maxrank} (cf.\;\cite[Thm.\;2.6]{Krao}, \cite{El}).
\end{remark}
  
\begin{remark} If $s= 4$ then the necessary condition $g+33 \ge 4d$ for $W$ to
  be an irreducible component implies $a\ge 4$ for $d > 16$ because $g =
  ad-2a^2+1$ in Theorem~\ref{mainQs}. By \cite[Rem.\;5.3]{KleOtt}, if
  $D:=C-4H=x\f_1+y\f_2$ is effective and $h^1(S,\sO_S(D))>0$, then either
  $\f_1$ is a fixed component of $|D|$ or $D$ is composed with a pencil. In
  the latter case we have $x=0$, i.e. $C=4\f_1+r\f_2$ for $r\ge 6$. 
  In a recent preprint (\cite{Nanew}) Nasu shows that these curves $C$ are
  obstructed for every $r \ge 6$ and in the case $r=6$ where
  $h^1(\sI_C(4))=1$, he shows that $W$ is a non-reduced component of
  $\HH(22,57)_{sc}$. {\sf More recently} we independently finished the case
  $(a,b)=(7,5)$, $s=4$ in the proof above for which $h^1(\sI_C(4))=2$ and
  $(d,g)$ attains the value $(d,g)=(22,57)$. Indeed our analysis also applies
  to show that $W$ is a non-reduced component in the case $r=6$. So there are
  at least 2 non-reduced components of $\HH(22,57)_{sc}$ with $s=4$ (and by
  Theorem~\ref{exis} one more non-reduced component for which $s=3$).
\end{remark}

\section{Non-reduced components of ${\rm H}(d,g)_{sc}$ for $s=3$}

Motivated by the Mumford's (\cite{Mu}) example of a non-reduced component, we
showed in \cite{K2} the existence of $3$-maximal families that form
non-reduced components of $\HH(d,\lfloor (d^2-4)/8 \rfloor)_{sc}$ for every $d
\ge 14$.
In \cite{K2} we also made a conjecture about non-reduced components when $s=3$.
A rough motivation for the conjecture is that the dimension of $s$-maximal
families $W(s):=W$ often seems to decrease with increasing $s$, thus making
the inclusion $W(s) \subset W(s')$ for $s'>s$ rare without having a particular
reason for such an inclusion to exist. 

\begin{conjecture} \label{conj} Let $W$ be a $3$-maximal 
  family of smooth connected, {\it linearly normal} space curves of degree $d
  > 9$ and genus $g$, whose general member $C$ is contained in a smooth cubic
  surface. Then $W$ is a non-reduced irreducible component of $ \HH(d,g)_{sc}$
  if and only if
 $$ d \geq 14, \ \ 3d - 18  \leq  g \leq  (d^2-4)/8 \ \ and  \ \
H^1(\sI_C(3))\neq 0 \, .$$
 \end{conjecture}

 Note that a 3-maximal family $W$ is closed and irreducible by our definition,
 and that $\dim W = d+g+18$ holds for $d > 9$. The above conjecture,
 originating in \cite{K2}, is here presented by modifications proposed by
 Ellia \cite{E} (see also \cite{DP} by Dolcetti, Pareschi), because they found
 counterexamples which heavily depended on the fact the general curves $C$ were
 {\it not} linearly normal (i.e. $H^1(\sI_C(1)) \neq 0$).

 The conjecture is known to be true in many cases. Indeed Mumford's 
 example of a non-reduced component is in the range of
 Conjecture~\ref{conj} (minimal with respect to both degree and genus). Also
 the main result by the author in \cite{K1} shows the
 conjecture provided 
\begin{equation} \label{Klegen}
g > 7 + (d-2)^2/8\ , \ \ d \geq 18 \ .
\end{equation}
Ellia shows in \cite{E} that Conjecture~\ref{conj} holds in the larger range
$g > G(d,5)$, $d \geq 21$ by first proving

 \begin{proposition} \label{Ellia} (Ellia) Let 
   $ d \geq 21$ and $g \geq 3d - 18$, let $W$ be a 3-maximal family of
   smooth connected space curves whose general curve $C$ sits on 
   a smooth surface and suppose that $ H^1(\sI_C(1)) = 0$. If $X$ is a
   generization of $C$ in $\HH(d,g)_{sc}$ satisfying $H^0(\sI_{X}(3)) = 0$,
   then $H^0(\sI_{X}(4)) = 0$.
\end{proposition}


See \cite[Prop.\;VI.2]{E} for a proof. More recently Nasu proves (and
reproves) a part of the conjecture by showing that the cup-product (i.e the
primary obstruction) 
is nonzero if $h^1(\sI_C(3))=1$ (\cite{N}).

When we try to show that generizations $X$ of a curve $C$ with $s(C)=s-1$ do
not exist, the hard part is usually the case $s(X)=s$. The non-existence of
such $X$ are taken care of by Proposition~\ref{Ellia} for $s=4$. Therefore
combining Ellia's result with semi-continuity arguments when $s(X) > s$, we
can extend the range where Conjecture~\ref{conj} holds in a good number of
cases. This is what we do in the Theorems~\ref{mainC} and \ref{4.7}. 
Recalling that we can associate a curve $C$ on $S$ and its corresponding
invertible sheaf $\sO_S(C)$ with a $7$-tuple of integers $(\delta,m_1,..,m_6)$
satisfying \eqref{numpic} below (by blowing up $\proj{2}$ in six general
points in the usual way, cf. \cite{GP2}), we first remark (cf. \cite[Lem.\;16
and Cor.\;17]{K1})

\begin{lemma} \label{vani} With notations as above it holds \\[-3mm]

  {\rm (i)} $d = 3\delta - \sum_{i=1 }^{6} \, m _{i} ~ , ~ ~ g = { \delta -1
    \choose 2} - \sum_{i=1 }^{6} \,{ m_{i} \choose 2} \, . $
\\[-3mm]

  {\rm (ii)} If \
  $g > (d^2-4)/8 $, then $ \ H^1(\sI_C(3)) = 0 \, ,$ whence $C$ is unobstructed. \\[-4mm]

  {\rm (iii)} If \ $ d \geq 14$ and $g \geq 3d-18$, then \ $ H^1(\sI_C(3)) \neq
  0 \ {\rm and} \ H^1(\sI_{C}(1))= 0 \ \Leftrightarrow \ 1 \leq m_6 \leq 2 \, .$
\end{lemma}

By Lemma~\ref{vani} (ii) the conditions of Conjecture~\ref{conj} are necessary
for $W$ to be a non-reduced component, and for the converse, and we may suppose
$m_6 = 1$ or $2$ by (iii).  If $m_6 = 1$ we recall

\begin {theorem} \label{mainC} Let $W$ be a 3-maximal family of smooth
  connected space curves, whose general member sits on a smooth cubic surface
  and corresponds to the $7$-tuple $(\delta,m_1,..,m_6)$ satisfying  \\[-3mm]
\begin{equation} \label{numpic} 
  \delta \geq m_1 \geq.. \geq m_6 \ \ and \ \ \delta \geq  m_1 + m_2 +m_3 \, . 
\end{equation}  \\[-5mm]
%
%
  Then $W$ is a non-reduced irreducible component of $\HH(d,g)_{sc}$
  provided \\[-3mm]

\hspace{0.2cm} a)  \ $m_6 = 1, \ m_5 \geq 6, \ d \geq 35$ \
and \ $(\delta, m_1,..,m_6) \neq (\lambda+18,\lambda+6,6,..,6,1)$ for any
$\lambda \geq 2$,
or  \\[-4mm]

\hspace{0.2cm} b) \ $m_6 = 1, m_5 = 5, m_ 4 \geq 7, \ d
\geq 35$ \ and \ $(\delta,
 m_1,..,m_6) \neq (\lambda+21,\lambda+7,7,..,7,5,1)$ for $\lambda \geq 2$.
\end{theorem}

For a proof, see the appendix to \cite[Thm.\;A.3]{KleOtt} by the first author.
For $m_6 = 2$,  $m_5 \ge 3$ see \cite{N}. \\[1mm]
It is also possible to use the idea of \cite[Sect.\,4]{K1} to
determine bounds for $\dim V$ where $V \supset W$ and $\dim V > \dim W$. This
is done in \cite[Sect.\,4]{KleOtt}, leading to \cite[Prop.\;A.7]{KleOtt} which
we will generalize and completely prove.
Note that by Proposition~\ref{Ellia} we can skip considering components with
$s(V)=4$.

\begin{theorem} \label{4.7} Let $W$ be a 3-maximal family of smooth
  connected space curves, whose general member is linearly normal and sits on
  a smooth cubic surface. If
\begin{equation}  \label{range4.7}
g   >   \max \,\{ \,{\frac{d ^{2} }{10}} - {\frac{d}{2}} +  18  ,\,  
G(d,t)\,\} \, ,  ~ d  \ge t^2-2t+2 \, ,
\end{equation}
for {\sf some} $t$ satisfying $6 \le t \le 8$, then $W$ is an irreducible
component of \ $\HH(d,g)_{sc}$. Moreover, $W$ is non-reduced if and only if
$H^1(\sI_C(3))\neq 0$. In particular Conjecture~\ref{conj} holds in the range
\eqref{range4.7}, e.g. if 
\begin{equation}  \label{range4.8}
g   >  {\frac{d ^{2} }{10}} - {\frac{d}{2}} +  18 \, ,  ~ d  \ge   54 \, .
\end{equation}
\end{theorem}

We have $G(d,6) \ge \frac{d ^{2} }{10} - {\frac{d}{2}} + 18 $ if and only if
$ d \le 74$. This is one of the reasons for proving \eqref{range4.7} not only
for $t = 6$, but also $t=7$ and 8 since it enlarges the range where the
conjecture
holds. 
Note also that Ellia proves the conjecture for $g > G(d,5)= {\frac{d ^{2}
  }{10}} + {\frac{d}{2}} + 1 - \frac{2r(5-r)}{5}$, 
so the improvement in Theorem~\ref{4.7} is only of order $d$. The improvement in
Theorem~\ref{mainC} is more substantial.

In the proof of Theorem~\ref{4.7} we will need (see \cite[Prop.\;4.4]{KleOtt}
for a proof):

\begin {proposition} \label{20} Let V be an irreducible component of
  $\HH(d,g)_{sc}$ whose general curve $C$ sits on some integral surface $F$ of
  degree $s \geq 4$. If 
  $d > s^2$, then
\begin{equation*}
  \dim\, V  \le \  {s+3 \choose 3}- 1 + \max \,\{ \,{\frac{d ^{2} }{s}}-g\,
  , \ {\frac{d ^{2} }{2s}}\, , \ (4-s)d+g-1+h^0(\sO_C(s-4)) \, \}\, .
\end{equation*}
\end{proposition}

\begin{proof} [Proof (of  Theorem~\ref{4.7})]
  To see that $W$ is an irreducible component, we suppose there exists a
  component $V$ of $\HH(d,g)_{sc}$ satisfying $W \subset V$ and $\dim W < \dim
  V$. Then $s(V) \ge 4$ by the definition of a 3-maximal family, whence $s(V)
  \ge 5$ by Proposition~\ref{Ellia}. Firstly suppose $g > G(d,6)$ and $d \ge
  26$. It follows that the general curve $X$ of $V$ satisfies $s:=s(X)=5$. To
  get a contradiction we will use Proposition~\ref{20} for $s = 5$ and the
  fact $\dim W =d+g+18$. Indeed since $X$ is a generization of a smooth
  connected linearly normal curve, it follows that a surface containing $X$ of
  the least possible degree is integral and moreover that $X$ is smooth,
  connected and linearly normal. 
  We have \\[-2mm]
  \[ d+g+18 < 55 + \max \,\{ \, \lfloor {\frac{d ^{2} }{5}}\rfloor -g\, , \
  \lfloor {\frac{d ^{2} }{10}} \rfloor \, , \ -d+g-1+4+h^1(\sI_X(1)) \,\}\,
  . \] Suppose the maximum to the right is attained by $ 55 -d+g-1+4$. Then
  $d+g+18 < 55 -d+g+3$ which is absurd since we have assumed $d>25$. Similarly
  if the maximum is attained by $55+\lfloor d ^{2}/5 \rfloor -g$ we get $ d +
  g+18 < 55 + \lfloor d ^{2}/5 \rfloor-g$ or equivalently $ 2g \le 36 +
  \lfloor d ^{2}/5 \rfloor -d$ which contradicts the assumption $g >
  \frac{d^2}{10}- \frac d 2 + 18 $. Also $d+ g+18 < 55 + \lfloor d ^{2}/10
  \rfloor $ leads to a contradiction because $36+\frac{d^2}{10}-d \le
  \frac{d^2}{10}- \frac d 2 + 18 $ for $d \ge 36$ and $36+\frac{d^2}{10}-d \le
  G(d,6)$ for $26 \le d < 36$. For the latter inequality, we remark that we,
  for $26 \le d < 31$, have to compute $G(d,6)$ in the $B$-range (or ``extended
  $C$-range'') using \eqref{maxgen3} and \eqref{maxgen4}. Thus we have
  proved that $W$ is an irreducible component of $\HH(d,g)_{sc}$.

  Secondly we suppose $g > G(d,7)$ and $d \ge 37$. It follows from the
  previous paragraph that the assumptions 
  $s(X)=5$ and $g > \frac{d^2}{10}- \frac d 2 + 18 $ yield a contradiction.
  Hence we may assume $s(X)=6$. Now if we use Proposition~\ref{20} for
  $s:=s(X) = 6$, which requires $d >36$,  we get \\[-2mm]
  \[ d+g+19 \le 83 + \max \,\{ \, \lfloor {\frac{d ^{2} }{6}}\rfloor -g\, , \
  \lfloor {\frac{d ^{2} }{12}} \rfloor \, , \ -2d+g-1+h^0(\sO_X(2)) \,\}\,
  . \] To estimate $h^0(\sO_X(2))$ we use Clifford's theorem to see
  $h^0(\sO_X(2))-1 \le \max \,\{ \, 2d-g,d\,\}$. Since it is easy to see
  \[ \max \,\{ \,32+ \frac 1 2 \lfloor {\frac{d ^{2} }{6}}\rfloor -\frac d 2
  \, , \ 64+ \lfloor {\frac{d ^{2} }{12}} \rfloor -d\,\}\, \le \frac{d^2}{10}-
  \frac d 2 + 18 \] for $d \ge 40$, we have a contradiction except for $37 \le
  d \le 39$. For these exceptions in the extended $C$-range we compute
  $G(d,7)$ by using \eqref{maxgen3} and \eqref{maxgen4} and we get a
  contradiction since $g > G(d,7)$ for $37 \le d \le 39$. Again we can
  conclude that $W$ is an irreducible component of $\HH(d,g)_{sc}$.

  Thirdly we suppose $g > G(d,8)$ and $d \ge 50$. It follows from the previous
  paragraph that the assumptions 
  $s(X)=6$ (or 5) and $g > \frac{d^2}{10}- \frac d 2 + 18 $ yield a
  contradiction. We may therefore assume $s(X)=7$. Now using
  Proposition~\ref{20} for
  $s:=s(X) = 7$, which requires $d >49$,  we get \\[-2mm]
  \[ d+g+19 \le 119 + \max \,\{ \, \lfloor {\frac{d ^{2} }{7}}\rfloor -g\, , \
  \lfloor {\frac{d ^{2} }{14}} \rfloor \, , \ -3d+g-1+h^0(\sO_X(3)) \,\}\,
  . \] We use Clifford's theorem to get $h^0(\sO_X(3))-1 \le \max \,\{ \,
  3d-g,3d/2 \,\}$. Since it is rather easy to see
  \[ \max \,\{ \,50+ \frac 1 2 \lfloor {\frac{d ^{2} }{7}}\rfloor -\frac d 2
  \, , \ 100+ \lfloor {\frac{d ^{2} }{14}} \rfloor -d\,\}\, \le
  \frac{d^2}{10}- \frac d 2 + 18 \] for $d \ge 46$, we have a contradiction,
  i.e. $W$ is an irreducible component of $\HH(d,g)_{sc}$.

  Since $G(d,8) \le \frac{d ^{2} }{10} - {\frac{d}{2}} + 18 $ for $d \ge 58$
  and $d \in \{ \, 54,55,56\}\,$, then $W$ is an irreducible component (again
  we need to work in the extended $C$-range). 

  Finally for $d=57$, $g= G(d,8)=315$ we check $\dim W \ge \dim V$
  directly. Indeed the general curve $X$ of $V$ is by \eqref{maxgen2} directly
  linked via a c.i.\;of type $(8,8)$ to a plane curve $X'$ of degree $r=7$
  since $d+r \equiv 0$ mod $8$. The mapping cone construction yields a minimal
  resolution $0 \rightarrow R(-15) \oplus R(-9) \rightarrow R(-8)^3 \to I(X)
  \rightarrow 0$,
    and since well known linkage formulas imply $h^1(\sO_X(8)) =
    h^0(\sI_{X'}(4))=20$ and $h^1(\sO_X(9)) = 10$, we get $\dim V= 4d +
    3h^1(\sO_X(8))-h^1(\sO_X(9))=285$ by \eqref{maxrank} while $\dim
    W=d+g+18=390$, a contradiction. Thus \eqref{range4.8} imply that $W$ is an
    irreducible component of $\HH(d,g)_{sc}$.

    Then using $ \dim W + h^1(\sI_C(3)) = h^0(\sN_C)$, cf. \eqref{alphaN}, we
    get the final statement and we are done.
  \end{proof}


What still lacks in this picture are existence results of components in the
range of the (proven part of the) conjecture. 
It is, however, not so
difficult to prove existence. One may, as Dolcetti and Pareschi do in
\cite{DP}, prove existence using Rathmann's work \cite{Ra}. 
By \cite{Ra} one knows that there
exists a smooth connected curve contained in a smooth del Pezzo surface in
$\PP^4$, of degree $d > 4$ and genus $g$ provided
\begin{equation} \label{Rat} 
 (d+12)\sqrt{d+9} -11d/2-35 \le g \le 1+ (d^2-4d)/8 \ .
\end{equation}
Indeed Rathmann shows the existence of a 
$6$-tuple $(\delta,m_1,..,m_5)$ satisfying $\sO_S(C)^2>0$ and $\delta \geq m_1
\geq.. \geq m_5 \ge 0 \ \ and \ \ \delta \geq m_1 + m_2
+m_3$ 
for every $(d,g)$ in the range \eqref{Rat}. Letting $m_6=0$ this leads to a
7-tuple $(\delta,m_1,..,m_5,m_6)$ satisfying \eqref{numpic} for every $(d,g)$
given as in Lemma~\ref{vani} (i) in the range \eqref{Rat}, whence we get the
existence of a smooth connected curve $C$ sitting on our smooth cubic surface
$S \subset \PP^3$. If we now add a $7$-tuple corresponding to $n$ hyperplanes,
$(3n,n,.,n)$, to $(\delta,m_1,..,m_5,0)$, the corresponding linear system will
contain smooth connected curves $X \in |C+nH|$ satisfying $m_6=n$ and
$H^1(I_X(v)) = 0$ for every non-negative integer $v \le n$. Using the
adjunction formula for the genus we easily see that the degree $d'$ and genus
$g'$ of $X$ satisfy
\[ d=d'-3n \ \ {\rm and} \ \ g = -nd'+g'+3(n^2+n)/2 \ .
\]
Inserting these formulas into \eqref{Rat} we get for every $d' > 3n+4$, $g'$ in
the range
\begin{equation} \label{RatX} (d'+12-3n)\sqrt{d'+9-3n} +d'(n-\frac {11} 2)-35+
  \frac {3(10n-n^2)}2 \le g' \le 1+ \frac {d'^2+(2n-4)d'-3n^2}8 \
\end{equation}
the existence of such a curve $X$. Using this for $n=1$ and $n=2$ we get the
following result.
\begin{theorem} \label{exis} For every $d \ge 14$ and $g$, $(d,g) \ne (14,22)$
  in the range
  \begin{equation} \label{RatX2} (d+9)\sqrt{d+6} -\frac {9d} 2-\frac {43} 2
    \le g \le \frac {d^2-4}8 \
\end{equation}
there exists a smooth connected curve $C$ of degree $d$ and genus $g$,
contained in a smooth cubic surface in $\PP^3$, whose corresponding $7$-tuple
$(\delta,m_1,..,m_6)$ has $m_6$ equal to 1 or 2 and such that
$(\delta,m_1,..,m_6) \ne (\lambda+ 3m_6,\lambda+ m_6,m_6,..,m_6)$ for every
$\lambda \ge 2$. If also $g \geq 3d-18$ holds, then $C$ satisfies
\[ H^1(\sI_C(3)) \neq 0 \ {\rm and} \ H^1(\sI_{C}(1))= 0\,. 
\] 
In particular for every $d, g \le \frac {d^2-4}8$ satisfying either
\eqref{Klegen}, \eqref{range4.7}, \eqref{range4.8} or $g>G(d,5), d\ge 21$
there exists a non-reduced component $W$ of \ $ \HH(d,g)_{sc}$ whose general
curve sits on a smooth cubic hypersurface. 
\end{theorem}
\begin{proof} Let \ $G1:= ${\Large$\frac {d^2-2d +5}{8}$},
  $G2:= ${\Large$\frac {d^2-4}8$}, and let
 $$ g1:=(d+9)\sqrt{d+6} -\frac {9d} 2-\frac {43} 2 \ \ , \ \ 
 g2:= (d+6)\sqrt{d+3} -\frac {7d} 2-11. $$ Using \eqref{RatX} with $n=1$
  (resp. $n=2$) we get the existence of a smooth connected curve $C$ for every
  $(d,g)$ in the range $g1 \le g \le G1$ (resp. $g2 \le g \le G2$) with a
  7-tuple with the desired properties. One verifies that $g1 \le g2 \le G2$
  for $d\ge 14$ and that $g2 < G1 <G2$ for $d \ge 17$. Since the only integer
  $k$ satisfying $G1 < k< g2$ for $14 \le d \le 16$ is $k=22$ in which case
  $d=14$, we get the first conclusion of the theorem. Moreover invoking
  Lemma~\ref{vani} (iii) we get the next conclusion. Finally looking at the
  statements accompanying \eqref{Klegen}, \eqref{range4.7}, \eqref{range4.8}
  or $g>G(d,5), d\ge 21$, we see that the conjecture holds in these ranges,
  and we get the final conclusion of the theorem.
\end{proof}
\begin{remark} {\rm (i)} \ For $d \ge 17$ we have $g1 \le g2 \le G1 \le G2$,
  with notations as in the proof. If $d, g$ satisfy $g2 \le g \le G1$ it
  follows from the proof that there exists two 7-tuples, one with $m_6=1$ and
  one with $m_6=2$ with $d, g$ satisfying \eqref{RatX2}. If we in addition are
  in a range of the proven part of the conjecture, we have two different
  3-maximal families that form non-reduced components of
  $\HH(d,g)_{sc}$. 
  {\rm (ii)} \ If we use \eqref{RatX} for $n=3$ we get the existence of a
  smooth connected curve $C$ for every $d \ge 14$ and $g$ in the range
  \begin{equation} (d+3)\sqrt{d} -\frac {5d} 2-\frac {7} 2 \le g \le \frac
    {d^2+2d-19}8 \ , 
  \end{equation}
  with $m_6=3$ and $(\delta, m_1,..,m_6) \neq (\lambda+9,\lambda+3,3,..3)$ for
  any $\lambda \geq 2$. Such $C$ belongs to a generically smooth, irreducible
  component of \ $\HH(d,g)_{sc}$ because $H^1(\sI_C(3)) = 0$. This improves upon
  the lower bound of \cite[Prop.\;3.1]{K4} for $d \gg 0$. Indeed
  \cite[Prop.\;3.1]{K4}, which states that for $d >9$ and $(d,g)$ satisfying
  \begin{equation} \label{GruPes} 3d-17+ \frac{(d-9)(d-18)}{18} \le g \le 1 +
    \frac {d(d-3)}{6} \ , \ \ \ and
  \end{equation}
  $(d,g) \notin \{(30,91), (33,103), (34,109) \}$ there exists an unobstructed
  curve $C$ of \ $\HH(d,g)_{sc}$ contained in a smooth cubic surface, gives a
  better lower bound for $9 < d \le 161$. The proof of \cite[Prop.\;3.1]{K4}
  which makes use of Gruson and Peskine's existence result
  \cite[Prop.\,2.10]{GP2}, produces a curve $C$ with 7-tuple satisfying $m_6
  \ge 3$ and such that $H^1(\sI_C(v)) = 0$ for $v \le 3$ for every $(d,g)$ in
  the mentioned range, thus $C$ belongs a 3-maximal family $W$ which is a
  generically smooth component of $\HH(d,g)_{sc}$. 
   \end{remark}

\bigskip \bigskip

\end{document}